\documentclass{amsart}
\usepackage{amssymb}
\usepackage{latexsym}
\usepackage{amsmath}
\usepackage{enumerate}

\newtheorem{theorem}{Theorem}[section]
\newtheorem{lemma}[theorem]{Lemma}
\newtheorem{corollary}[theorem]{Corollary}
\newtheorem{proposition}[theorem]{Proposition}

\theoremstyle{definition}
\newtheorem{definition}[theorem]{Definition}

\newtheorem{fact}[theorem]{Fact}

\newtheorem{claim}[theorem]{Claim}

\newtheorem*{convention}{Convention}

\numberwithin{equation}{section}

\title[Short Title]{The strong reflecting property and Harrington's Principle}

\author{Yong Cheng}
\address{Institut f\"ur mathematische Logik und Grundlagenforschung, Universit\"at M\"unster, Einsteinstr. 62, 48149  M\"unster, Germany}
\email{world-cyr@hotmail.com}

\thanks{Some materials of this paper are evolved from the author's Ph.D. thesis written in 2012 at the National University of Singapore under the supervision of Chong Chi Tat and W.Hugh Woodin. I would like to thank W.Hugh Woodin for his support and guidance on the thesis.  I would like to thank members of my Ph.D. committee. I would like to thank Ralf Schindler for his support through SFB 878. I would like to thank  referees for their careful reading and helpful comments.}



\begin{document}

\begin{abstract}
In this paper we characterize the strong reflecting property for $L$-cardinals for all $\omega_n$, characterize Harrington's Principle $HP(L)$ and its generalization and discuss the relationship between the strong reflecting property for $L$-cardinals and  Harrington's Principle $HP(L)$.
\end{abstract}

\maketitle

\section{Introduction and preliminaries}

The notion of the strong reflecting property for $L$-cardinals is introduced in \cite[Definition 2.8]{YongCheng1}. The motivation of introducing this notion is to force a set model of Harrington's Principle, $HP(L)$ for short (cf. Definition \ref{def of HP}), over higher order arithmetic (cf. Definition \ref{def of higher order ar}). However the proof of The Main Theorem in \cite{YongCheng1} uses very little knowledge about the strong reflecting property for $L$-cardinals. In this paper, in Section 2 we develop the full theory of the strong reflecting property for $L$-cardinals and characterize $SRP^{L}(\omega_n)$ for $n\in\omega$ (see Proposition \ref{srp of omega first}, Proposition \ref{second strong reflecting property}, Theorem \ref{strength of omega two} and Theorem \ref{characterization of srp above third level}). We also generalize some results on $SRP^{L}(\gamma)$ to $SRP^{M}(\gamma)$ for other inner  models $M$ (see Theorem \ref{characterizaton of zero sharp thm} and Theorem \ref{big thm about u}).

In Section 3, we define the generalized
Harrington's Principle $HP(M)$ for any inner model $M$, give
characterizations of $HP(M)$ for some well known inner models (see Theorem \ref{cited theorem from Woodin} and \ref{HP for measurable cn}) and show
that, in some cases, this generalized principle fails (see Corollary \ref{Theorem for core model} and Theorem \ref{HP for HOD}). In Section 4, we discuss the relationship between the strong reflecting property for $L$-cardinals and  Harrington's Principle $HP(L)$.

Our definitions and notations are standard. We refer to textbooks such as \cite{Jech}, \cite{Higherinfinite} and \cite{Kunen1} for the definitions and notations we use. For the definition of admissible set and admissible ordinal, see \cite{Constructiblity}. For notions of large cardinals, see \cite{Higherinfinite}. Our notations about forcing are standard (cf. \cite{Jech} and \cite{JamesCummings}). For the  theory of $0^{\sharp}$ see \cite{Constructiblity} and \cite{Jech}. Recall that $0^{\sharp}$ is the unique well founded remarkable $E.M.$ set, and $0^{\sharp}$ exists if and only if for some uncountable limit ordinal $\lambda, L_{\lambda}$ has an uncountable set of indiscernibles (cf. \cite{Constructiblity} and \cite{Jech}). For the  theory of $0^{\dag}$ see \cite{Higherinfinite}.

\begin{definition}\label{def of higher order ar}
(\cite{YongCheng1})\begin{enumerate}[(i)]
  \item $Z_{2}= ZFC^{-} +$ Any set is Countable.\footnote{$ZFC^{-}$ denotes $ZFC$  with the Power Set Axiom deleted and Collection instead of Replacement. For the discussion of the theory $ZFC$ without power set, see \cite{VictoriaGitman}.}
  \item $Z_{3}= ZFC^{-} + \mathcal{P}(\omega)$ exists + Any set is of cardinality $\leq \beth_1$.
      \item $Z_4=ZFC^{-}+ \mathcal{P}(\mathcal{P}(\omega))$ exists + Any set is of cardinality $\leq \beth_2$.
\end{enumerate}
\end{definition}

$Z_2, Z_3$ and $Z_4$ are the corresponding axiomatic systems for Second Order Arithmetic (SOA), Third Order Arithmetic and Fourth Order Arithmetic.

Throughout this paper whenever we write $X\prec H_{\kappa}$ and $\gamma\in X$, $\bar{\gamma}$ always denotes the image of $\gamma$ under the transitive collapse of $X$. If $U$ is an ultrafilter on $\kappa$, we say that $U$ is countably complete if and only if whenever $Y\subseteq U$ is countable, we have that $\bigcap Y\neq\emptyset$. The distinction between $V$-cardinals and $L$-cardinals is present throughout the article. Whenever we write $\omega_n$ (for some $n$) without a superscript it is understood that we mean the $\omega_n$ of $V$. In this paper, $\kappa$-model is a model in the form $L[U]$ such that $\langle L[U], \in, U\rangle\models U$ is a normal ultrafilter over $\kappa$.

\section{Characterizations of the strong reflecting property for $L$-cardinals}

In this section we develop the full theory of the strong reflecting property for $L$-cardinals and characterize $SRP^{L}(\omega_n)$ for $n\in\omega$. We also generalize some results on $SRP^{L}(\gamma)$ to $SRP^{M}(\gamma)$ for any inner  model $M$.

Recall that an inner model $M$ is $L$-like if $M$ is in the form $\langle L[\vec{E}], \in, \vec{E} \rangle$ where $\vec{E}$ is a coherent sequence of
extenders; moreover, for an $L$-like inner model $M, M|\theta$ is of the form $\langle J_{\theta}^{\vec{E}}, \in,\vec{E}\upharpoonright \theta, \varnothing\rangle$.\footnote{For the definition of coherent sequences of extenders $\vec{E}$, $J_{\alpha}^{\vec{E}}$ and $\vec{E}\upharpoonright \alpha$, see Section 2.2 in \cite{Steel}.}

\begin{convention}\label{convention}
Throughout, whenever we consider an inner model $M$ we assume that $M$ is $L$-like and has the property that $M|\theta$ is definable in $H_{\theta}$ for any regular cardinal $\theta>\omega_2$.\footnote{All known core models satisfy this convention.}
\end{convention}

\begin{definition}\label{definition of strong reflecting property and weakly}
Let $\gamma\geq\omega_1$ be an $L$-cardinal.
\begin{enumerate}[(i)]
  \item $\gamma$ has the strong reflecting property for $L$-cardinals, denoted $SRP^{L}(\gamma)$, if and only if for some  regular cardinal $\kappa> \gamma$,  if $X\prec H_{\kappa}, |X|=\omega$ and $\gamma\in X$, then $\bar{\gamma}$ is an $L$-cardinal.
  \item $\gamma$ has the weak reflecting property for $L$-cardinals, denoted $WRP^{L}(\gamma)$,  if and only if for some regular cardinal $\kappa> \gamma$, there is  $X\prec  H_{\kappa}$ such that $|X|=\omega, \gamma\in X$ and $\bar{\gamma}$ is an $L$-cardinal.
\end{enumerate}
\end{definition}

\begin{proposition}\label{reflecting}
Suppose $\gamma\geq\omega_1$ is an $L$-cardinal. Then the  following are equivalent:
\begin{enumerate}[(1)]
  \item $SRP^{L}(\gamma)$.
  \item For any regular cardinal $\kappa> \gamma$, if $X\prec H_{\kappa}, |X|=\omega$ and $\gamma\in X$, then $\bar{\gamma}$ is an $L$-cardinal.
      \item For some  regular cardinal $\kappa> \gamma, \{X\mid X\prec H_{\kappa},  |X|=\omega, \gamma\in X$ and $\bar{\gamma}$ is an $L$-cardinal\} contains a club.
      \item  There exists $F:\gamma^{<\omega}\rightarrow\gamma$ such that if  $X\subseteq\gamma$ is countable and  closed under $F$,\footnote{In this paper, we say that $X$ is closed under $F$ if $F``X^{<\omega}\subseteq X$.} then $o.t.(X)$ is an $L$-cardinal.
          \item For any  regular cardinal $\kappa> \gamma, \{X\mid X\prec H_{\kappa},  |X|=\omega, \gamma\in X$ and $\bar{\gamma}$ is an $L$-cardinal\} contains a club.
\end{enumerate}
\end{proposition}
\begin{proof}
Note that $(2)\Rightarrow (1), (1)\Rightarrow (3), (2)\Rightarrow (5)$ and $(5)\Rightarrow (3)$. It suffices to show that $(4)\Rightarrow (2)$ and $(3)\Rightarrow (4)$. For the proof see \cite[Proposition 2.7]{YongCheng1}.
\end{proof}

Suppose $\gamma\geq\omega_1$ is an $L$-cardinal. Let $(1)^{\ast}, (2)^{\ast}, (3)^{\ast},(4)^{\ast}$ and $(5)^{\ast}$ respectively be the statements which replace ``is an $L$-cardinal" with ``is not an $L$-cardinal" in Definition \ref{definition of strong reflecting property and weakly}(i) and  statements $(2), (3),(4)$ and $(5)$ in Proposition \ref{reflecting}. The following corollary is an observation from the proof of Proposition \ref{reflecting}.

\begin{corollary}\label{corollary about strong reflecting property}
$(1)^{\ast}\Leftrightarrow (2)^{\ast}\Leftrightarrow (3)^{\ast}\Leftrightarrow (4)^{\ast}\Leftrightarrow (5)^{\ast}$.
\end{corollary}

\begin{proposition}\label{equivalent forms of strong reflecting property}
Suppose $\gamma\geq\omega_1$ is an $L$-cardinal, $\kappa$ is regular  and $|\gamma|=\kappa$. Then the  following are equivalent:
\begin{enumerate}[(a)]
\item $SRP^{L}(\gamma)$.
  \item For any bijection $\pi: \kappa\rightarrow \gamma$, there exists a club $D\subseteq\kappa$ such that for any $\theta\in D$, $o.t.(\{\pi(\alpha)\mid\alpha< \theta\})$ is an $L$-cardinal.
      \item For some bijection $\pi: \kappa\rightarrow \gamma$, there exists a club $D\subseteq\kappa$ such that for any $\theta\in D$, $o.t.(\{\pi(\alpha)\mid\alpha<\theta\})$ is an $L$-cardinal.
\end{enumerate}
\end{proposition}
\begin{proof}
The proof is essentially the same as the case $\kappa=\omega_1$ in \cite[Proposition 2.9]{YongCheng1}.
\end{proof}

Let $(6)^{\ast}$ and $(7)^{\ast}$ respectively be the statement which replaces ``is an $L$-cardinal" with ``is not an $L$-cardinal" in Proposition \ref{equivalent forms of strong reflecting property}(b) and Proposition \ref{equivalent forms of strong reflecting property}(c). The following corollary is an observation from the proof of Proposition \ref{equivalent forms of strong reflecting property}.

\begin{corollary}\label{corollary about weakly reflecting property}
Suppose $\gamma\geq\omega_1$ is an $L$-cardinal, $\kappa$ is regular  and $|\gamma|=\kappa$. Then $(1)^{\ast}\Leftrightarrow  (6)^{\ast}\Leftrightarrow  (7)^{\ast}$.
\end{corollary}

\begin{proposition}\label{weakly reflecting}
Suppose $\gamma\geq\omega_1$ is an $L$-cardinal. Then the following are equivalent:
\begin{enumerate}[(a)]
  \item $WRP^{L}(\gamma)$.
  \item For any   regular cardinal $\kappa> \gamma$, there is $X\prec  H_{\kappa}$ such that $|X|=\omega, \gamma\in X$ and $\bar{\gamma}$ is an $L$-cardinal.
       \item For some   regular cardinal $\kappa> \gamma, \{X\mid X\prec H_{\kappa},  |X|=\omega, \gamma\in X$ and  $\bar{\gamma}$ is an $L$-cardinal\} is stationary.
      \item For any $F: \gamma^{<\omega}\rightarrow\gamma$, there exists $X\subseteq\gamma$ such that $X$ is countable, closed under $F$ and $o.t.(X)$ is an $L$-cardinal.
           \item For any regular cardinal $\kappa> \gamma, \{X\mid X\prec H_{\kappa},  |X|=\omega, \gamma\in X$ and $\bar{\gamma}$ is an $L$-cardinal\} is stationary.
\end{enumerate}
\end{proposition}
\begin{proof}
Note that $(e)\Rightarrow (c)$ and $(c)\Rightarrow (a)$. It suffices to show that $(a)\Rightarrow (d), (d)\Rightarrow (b)$ and $(b)\Rightarrow (e)$. $(a)\Rightarrow (d)$ follows from $(4)^{\ast}\Leftrightarrow (2)^{\ast}$ in Corollary \ref{corollary about strong reflecting property}. $(d)\Rightarrow (b)$ follows from $(1)^{\ast}\Leftrightarrow (4)^{\ast}$ in Corollary \ref{corollary about strong reflecting property}. $(b)\Rightarrow (e)$ follows from $(3)^{\ast}\Leftrightarrow (1)^{\ast}$ in Corollary \ref{corollary about strong reflecting property}.
\end{proof}

\begin{proposition}
Suppose $\gamma\geq\omega_1$ is an $L$-cardinal, $\kappa$ is regular  and $|\gamma|=\kappa$. Then the following are equivalent:
\begin{enumerate}[(1)]
\item $WRP^{L}(\gamma)$.
\item For some bijection $\pi: \kappa\rightarrow \gamma$, there exists a stationary $D\subseteq\kappa$ such that for any $\theta\in D$, $o.t.(\{\pi(\alpha)\mid\alpha<\theta\})$ is an $L$-cardinal.
    \item For any bijection $\pi: \kappa\rightarrow \gamma$, there exists a stationary $D\subseteq\kappa$ such that for any $\theta\in D$, $o.t.(\{\pi(\alpha)\mid\alpha< \theta\})$ is an $L$-cardinal.
\end{enumerate}
\end{proposition}
\begin{proof}
Follows from Corollary \ref{corollary about weakly reflecting property} and $(1)^{\ast}\Leftrightarrow (2)^{\ast}$ in Corollary \ref{corollary about strong reflecting property}. The proof is standard and we omit the details.
\end{proof}

\begin{proposition}\label{srp of omega first}
The following are equivalent:
\begin{enumerate}[(1)]
  \item $\omega_1$ is a limit cardinal in $L$.
  \item $WRP^{L}(\omega_1)$.
  \item $SRP^{L}(\omega_1)$.
\end{enumerate}
\end{proposition}
\begin{proof}
It suffices to show that $(1)\Rightarrow (3)$ and $(2)\Rightarrow (1)$ since $(3)\Rightarrow (2)$ is immediate.

$(1)\Rightarrow (3)$ Suppose $\omega_1$ is a limit cardinal in $L$. Then $\{\alpha<\omega_1: \alpha$ is an $L$-cardinal\} is a club. By Proposition \ref{equivalent forms of strong reflecting property}, $SRP^{L}(\omega_1)$ holds.

$(2)\Rightarrow (1)$ Suppose $WRP^{L}(\omega_1)$ holds.  Then $\{X\cap \omega_1| X\prec H_{\omega_2}\wedge |X|=\omega\wedge o.t.(X\cap\omega_1)$ is an $L$-cardinal\} is stationary in $\omega_1$. It is easy to see that for any $\alpha<\omega_1$ there is $\alpha<\beta<\omega_1$ such that $\beta$ is an $L$-cardinal.
\end{proof}

\begin{proposition}\label{transive collapse of L cardinal}
Suppose $\gamma\geq\omega_1$ is an $L$-cardinal, $\kappa>\gamma$ is a regular cardinal and $SRP^{L}(\gamma)$ holds. If $Z\prec H_{\kappa}$, $|Z|\leq \omega_1$ and $\gamma\in Z$, then $\bar{\gamma}$ is an  $L$-cardinal.
\end{proposition}
\begin{proof}
Suppose $\bar{\gamma}$ is not an $L$-cardinal. Let $M$ be the transitive collapse of $Z$ and $\pi: M\prec H_{\kappa}$ be the inverse of the collapsing map. Take $Y\prec H_{\kappa}$ such that $|Y|=\omega$ and $M, \bar{\gamma} \in Y$. Note that $Y\models ``\bar{\gamma}$ is not an $L$-cardinal". Hence $\bar{\bar{\gamma}}$ is not an $L$-cardinal.\footnote{$\bar{\bar{\gamma}}$ is the image of $\bar{\gamma}$ under the transitive collapse of $Y$.} Let $X=\pi``(Y\cap M)$. Since $\bar{\gamma}\in Y\cap M$ and $\pi(\bar{\gamma})=\gamma$, $\gamma\in X$. Note that $X\prec Z\prec H_{\kappa}$ and the image of $\gamma$ under the transitive collapse of $X$ is $\bar{\bar{\gamma}}$. By $SRP^{L}(\gamma)$, $\bar{\bar{\gamma}}$ is an $L$-cardinal. Contradiction.
\end{proof}

\begin{proposition}\label{compare strong cardinal}
Suppose $\omega_1\leq\gamma_0<\gamma_1$ are $L$-cardinals. Then $SRP^{L}(\gamma_1)$ implies $SRP^{L}(\gamma_0)$ (respectively $WRP^{L}(\gamma_1)$ implies $WRP^{L}(\gamma_0)$).
\end{proposition}
\begin{proof}
We only show the strong reflecting property case (the argument for the weak reflecting  property case is similar). Let $\kappa>\gamma_1$ be a regular cardinal. It suffices to show if $X\prec H_{\kappa}, |X|=\omega$ and $\{\gamma_0, \gamma_1\}\subseteq X$, then $\bar{\gamma_0}$ is an $L$-cardinal. Note that $L_{\gamma_1}\models \gamma_0$ is a cardinal. Since $\gamma_1 \in X, L_{\gamma_1}\in X$. Since $\bar{L_{\gamma_1}}=L_{\bar{\gamma_1}}$ and $\bar{L_{\gamma_1}}\models \bar{\gamma_0}$ is a cardinal, $L_{\bar{\gamma_1}}\models \bar{\gamma_0}$ is a cardinal. By $SRP^{L}(\gamma_1)$, $\bar{\gamma_1}$ is an $L$-cardinal and hence $\bar{\gamma_0}$ is an $L$-cardinal.
\end{proof}

\begin{proposition}\label{second strong reflecting property}
The following are equivalent:
\begin{enumerate}[(1)]
  \item $SRP^{L}(\omega_2)$.
  \item $\omega_2$ is a limit cardinal in $L$ and for any $L$-cardinal $\omega_1\leq\gamma<\omega_2$, $SRP^{L}(\gamma)$ holds.
      \item $\{\alpha<\omega_2\mid\alpha$ is an $L$-cardinal and $SRP^{L}(\alpha)$ holds\} is unbounded in $\omega_2$.
\end{enumerate}
\end{proposition}
\begin{proof}
$(1)\Rightarrow (2)$ By Proposition \ref{compare strong cardinal}, it suffices to show $\omega_2$ is a limit cardinal in $L$. Let $\kappa>\omega_2$ be the regular cardinal that witnesses $SRP^{L}(\omega_2)$. Fix $\alpha<\omega_2$. Pick $Z\prec H_{\kappa}$ such that $|Z|=\omega_1$, $\alpha\subseteq Z$ and $\omega_2\in Z$. By Proposition \ref{transive collapse of L cardinal}, $\bar{\omega_2}$ is an $L$-cardinal. Note that $\alpha\leq \bar{\omega_2}<\omega_2$.

$(2)\Rightarrow (1)$ Suppose $\kappa>\omega_2$ is a regular cardinal, $X\prec H_{\kappa}, |X|=\omega$ and $\omega_2\in X$. We show that $\bar{\omega_2}$ is an $L$-cardinal. Note that $\bar{\omega_2}=o.t.(X\cap \omega_2)$. Let $E=\{\gamma\mid \omega_1\leq\gamma<\omega_2 \wedge\gamma$ is an $L$-cardinal\}. $E$ is definable in  $H_{\kappa}$. Since $\omega_2$ is a limit cardinal in $L$, $E$ is cofinal in $\omega_2$ and hence $E\cap X$ is cofinal in $\omega_2\cap X$. For  $\gamma\in E\cap X, \bar{\gamma}=o.t.(X\cap\gamma)$ and by $SRP^{L}(\gamma)$, $\bar{\gamma}$ is an $L$-cardinal. Note that $\bar{\omega_2}=sup(\{\bar{\gamma}\mid \gamma\in E\cap X\})$. Hence $\bar{\omega_2}$ is an $L$-cardinal.

$(1)\Leftrightarrow (3)$ Follows from $(1)\Leftrightarrow (2)$ and Proposition \ref{compare strong cardinal}.
\end{proof}

The notion of remarkable cardinal is introduced by Ralf Schindler in \cite{Schindler2}. Any remarkable cardinal is remarkable in $L$ (cf.\cite[Lemma 1.7]{Schindler2}).

\begin{definition}
(\cite{Schindler2})
\begin{enumerate}[(1)]
  \item Let $\kappa$ be a cardinal, $G$ be $Col(\omega, <\kappa)$-generic over $V$, $\theta>\kappa$ be a regular cardinal and $X\in [H_{\theta}^{V[G]}]^{\omega}$.  We say that $X$ condenses remarkably if $X=ran(\pi)$ for some elementary $\pi: (H_{\beta}^{V[G\cap H_{\alpha}^{V}]}, \in, H_{\beta}^{V}, G\cap H_{\alpha}^{V})\rightarrow (H_{\theta}^{V[G]}, \in, H_{\theta}^{V}, G)$ where $\alpha=crit(\pi)<\beta<\kappa$ and $\beta$ is a cardinal in $V$.
          \item  For regular cardinal $\theta>\kappa, \kappa$ is $\theta$-remarkable if and only if in $V^{Col(\omega, <\kappa)}, \{X\in [H_{\theta}]^{\omega}: X$  condenses remarkably\} is stationary. We say that $\kappa$ is remarkable if $\kappa$ is $\theta$-remarkable for all regular cardinal $\theta>\kappa$.
\end{enumerate}
\end{definition}

\begin{lemma}\label{key lemma on relationship between remark and wfp}
(\cite[Lemma 2.3]{YongCheng1})\quad Suppose $\kappa$ is an $L$-cardinal. The following are equivalent:
\begin{enumerate}[(1)]
  \item $\kappa$ is remarkable in $L$;
  \item  If $\gamma\geq\kappa$ is an $L$-cardinal, $\theta>\gamma$ is  a regular cardinal in $L$, then $\Vdash^{L}_{Col(\omega, <\kappa)} ``\{X| X\prec L_{\check{\theta}}[\dot{G}], |X|=\omega$ and $o.t.(X\cap\check{\gamma})$ is an $L$-cardinal\} is stationary".
\end{enumerate}
\end{lemma}

\begin{corollary}\label{key coro of above lemma}
If $\kappa$ is remarkable in $L$ and $G$ is $Col(\omega, <\kappa)$-generic over $L$, then $L[G]\models WRP^{L}(\gamma)$ holds for any $L$-cardinal $\gamma\geq\kappa$.
\end{corollary}
\begin{proof}
Follows from Lemma \ref{key lemma on relationship between remark and wfp}.
\end{proof}

Fix some $L$-cardinal $\gamma\geq\omega_1$. $SRP^{L}(\gamma)$ is upward absolute (cf. \cite[Proposition 2.11]{YongCheng1}).\footnote{The key point is that the statement Proposition \ref{reflecting}(4) is upward absolute.} As a corollary, $WRP^{L}(\gamma)$ is downward absolute.\footnote{The key point is that the statement Proposition \ref{weakly reflecting}(d) is downward absolute.} So if $WRP^{L}(\gamma)$ holds, then $WRP^{L}(\gamma)$ holds in $L$. The converse is not true in general.

\begin{proposition}\label{new non propro}
Suppose $WRP^{L}(\kappa)$ holds where $\kappa\geq\omega_1$ is an $L$-cardinal. Then $L\models\omega_1$ is $\kappa^{+}$-remarkable and for any regular $\theta>\kappa$ in $L$, $L\models\omega_1$ is $\theta$-remarkable.
\end{proposition}
\begin{proof}
$L\models WRP^{L}(\kappa)$ \text{if{f}} $\{X| X\prec L_{\kappa^{+}}, |X|=\omega$ and $o.t.(X\cap\kappa)$ is an $L$-cardinal\} is stationary in $L$ \text{if{f}} for any $L$-regular cardinal $\theta>\kappa, \{X| X\prec L_{\theta}, |X|=\omega$ and $o.t.(X\cap\kappa)$ is an $L$-cardinal\} is stationary in $L$. For $L$-regular cardinal $\theta>\kappa$, $L\models\omega_1$ is $\theta$-remarkable \text{if{f}} for any $G$ which is $Col(\omega, <\omega_1)$-generic over $L$, $L[G]\models \{X\in [L_{\theta}]^{\omega}| X=ran(\pi), \pi: (L_{\beta}[G\upharpoonright\alpha],\in,L_{\beta},G\upharpoonright\alpha)\prec (L_{\theta}[G],\in, L_{\theta},G)$ where $\alpha=crit(\pi)<\beta<\omega_1$ and $\beta$ is an $L$-cardinal\} is stationary. Note that $L\models WRP^{L}(\kappa)$ and $Col(\omega, <\omega_1)$ is stationary preserving.
\end{proof}

\begin{corollary}
``For any $L$-cardinal $\gamma\geq\omega_1, WRP^{L}(\gamma)$ holds" is equiconsistent with $\omega_1$ is remarkable.
\end{corollary}
\begin{proof}
Follows from Corollary \ref{key coro of above lemma} and Proposition \ref{new non propro}.
\end{proof}

\begin{theorem}\label{strength of omega two}
(Set forcing)\quad The following two theories are equiconsistent:
\begin{enumerate}[(1)]
  \item $SRP^{L}(\omega_2)$.
  \item $ZFC\, +$ there exists a remarkable cardinal with a weakly inaccessible cardinal above it.
\end{enumerate}
\end{theorem}
\begin{proof}
We first show that the consistency of (2) implies the consistency of (1). Let $S=\{\omega_1\leq\alpha<\omega_2\mid \alpha$ is an $L$-cardinal\}. Note that $SRP^{L}(\omega_2)$ is equivalent to $S$ being a club such that $SRP^{L}(\alpha)$ holds for any $\alpha\in S$. In \cite[Section 3.1]{YongCheng1}, assuming there exists a remarkable cardinal with a weakly inaccessible cardinal above it, we force a model $L[G,H]$ in which $S$ is a club and $SRP^{L}(\alpha)$ holds for any $\alpha\in S$. So $SRP^{L}(\omega_2)$ holds in $L[G,H]$.

From \cite[Section 3.2-3.4]{YongCheng1}, if $S$ is a club and $SRP^{L}(\alpha)$ holds for any $\alpha\in S$, then we can force a model of $Z_3\, +\, HP(L)$ . So the consistency of (1) implies the consistency of $Z_3 + HP(L)$. By \cite[Theorem 3.2]{YongCheng2}, $Z_3\, +\, HP(L)$  implies $L\models ZFC\, +\, \omega_1^{V}$ is remarkable. By Proposition \ref{second strong reflecting property}, $\omega_2^{V}$ is inaccessible in $L$. So the consistency of (1) implies the consistency of (2).
\end{proof}

\begin{definition}\label{general definition of strong reflecting property and weakly}
Suppose $M$ is an inner model and $\gamma\geq\omega_1$ is an $M$-cardinal. We say that $\gamma$ has the strong reflecting property for $M$-cardinals, denoted $SRP^{M}(\gamma)$, if and only if for some  regular cardinal $\kappa> \gamma$,  if $X\prec H_{\kappa}, |X|=\omega$ and $\gamma\in X$, then $\bar{\gamma}$ is an $M$-cardinal.
\end{definition}

\begin{definition}\label{covering thm for mc}
Suppose $M$ is an inner model. We say that $M$ has the full covering property if for any set $X$ of ordinals, there is $Y\in M$ such that $X\subseteq Y$ and $|Y|=|X|+\omega_1$. We say that $M$ has the rigidity property if there is no nontrivial elementary embedding from $M$ to $M$.
\end{definition}

\begin{theorem}\label{characterizaton of zero sharp thm}
Suppose $M$ is an inner model which satisfies Convention \ref{convention} and has both the full covering and the rigidity property. Then, for every M-cardinal $\gamma>\omega_2, SRP^{M}(\gamma)$ fails.
\end{theorem}
\begin{proof}
Suppose $SRP^{M}(\gamma)$ holds for some $\gamma>\omega_2$. Let $\kappa>\gamma$ be the witnessing  regular cardinal for $SRP^{M}(\gamma)$. Build an elementary chain $\langle Z_{\alpha}\mid\alpha<\omega_1\rangle$ of submodels of $H_{\kappa}$ such that for all $\alpha<\beta<\omega_1$, $Z_{\alpha}\prec Z_{\beta}\prec H_{\kappa}$, $Z_{\alpha}\in Z_{\beta}$ , $|Z_{\alpha}|=\omega$ and $\{\gamma, \omega_2\} \subseteq Z_{0}$.

Let $Z=\bigcup_{\alpha<\omega_1} Z_{\alpha}$. Then $|Z|=\omega_1$ and $Z\prec H_{\kappa}$. Let $\pi: N\cong Z\prec H_{\kappa}$ and $\pi_{\alpha}: N_{\alpha}\cong Z_{\alpha}\prec H_{\kappa}$ be the inverses of the collapsing maps. Let $j_{\alpha}: N_{\alpha}\prec N$ be the induced elementary embedding.  Since $\omega_1\subseteq Z$, $crit(\pi)>\bar{\omega_1}$. Since $\omega_2\in Z$ and $|Z|=\omega_1, crit(\pi)\leq \bar{\omega_2}$. So $crit(\pi)=\bar{\omega_2}$.

Note that Proposition \ref{transive collapse of L cardinal} still holds if we replace $L$ with $M$. By $SRP^{M}(\gamma)$, $\bar{\gamma}$ is an $M$-cardinal. Since $M|\bar{\gamma}$ is definable in $H_{\kappa}$, $\mathcal{P}(\bar{\omega_2})\cap M\subseteq M|\bar{\gamma}\in N$ and $\mathcal{P}(\bar{\omega_2})\cap M\in N$. Define $U=\{X\subseteq \bar{\omega_2}\mid X\in M\wedge\bar{\omega_2}\in \pi(X)\}$. $U$ is an $M$-ultrafilter. For  $\alpha<\omega_1$, the image of $Z_{\alpha}$ under the transitive collapse of $Z$ is $j_{\alpha}``N_{\alpha}$ and $j_{\alpha}``N_{\alpha}\in N$.

\begin{lemma}\label{a key lemma on U}
$U$ is countably complete.
\end{lemma}
\begin{proof}
Suppose $Y\subseteq U$ and $Y$ is countable. We show that $\bigcap Y\neq\emptyset$. Since $Y\subseteq N$, take $\alpha<\omega_1$ large enough such that $Y\subseteq j_{\alpha}``N_{\alpha}$. Let $S=\mathcal{P}(\bar{\omega_2})\cap M\cap j_{\alpha}``N_{\alpha}$.  Note that $S\in N$ and $N\models S$ is countable.

Note that $H_{\kappa}\models ``M$ has the full covering property"\footnote{Here we use that $M|\theta$ is definable in $H_{\theta}$ for regular cardinal $\theta>\omega_2$.} and hence  $N\models M$ has the full covering property.  Fix $T\in N$ such that $T\subseteq \mathcal{P}(\bar{\omega_2})\cap M, T\supseteq S, T\in M$ and $N\models |T|=\omega_1$. Since $\bar{\omega_2}=crit(\pi)>\omega_1, \pi(T)=\pi``T$. Since $T\in N, \mathcal{P}(T)\cap M\in N$.

\begin{claim}
$U\cap T\in N$.
\end{claim}
\begin{proof}
Since $\pi(T)=\pi``T\in M, \pi``(U\cap T)=\{\pi(A)\mid A\in T\wedge \bar{\omega_2}\in \pi(A)\}=\{B\in \pi(T)\mid \bar{\omega_2}\in B \}$ and $\pi``(U\cap T)\in M$.  Note that $\mathcal{P}(\pi``T)\cap M=\pi``(\mathcal{P}(T)\cap M)$ since for all $D\in \mathcal{P}(T)\cap M, \pi(D)=\pi``D$. Since $\pi``(U\cap T)\in \mathcal{P}(\pi``T)\cap M$, $\pi``(U\cap T)=\pi(D)=\pi``D$ for some $D\in \mathcal{P}(T)\cap M\subseteq N$. So $U\cap T=D$ and hence $U\cap T\in N$.
\end{proof}

Note that $Y\subseteq j_{\alpha}``N_{\alpha}\cap \mathcal{P}(\bar{\omega_2})\cap M=S\subseteq T$. Since $Y\subseteq T\cap U$, to show that $\bigcap Y\neq\emptyset$, it suffices to show that $\bigcap (U\cap T)\neq\emptyset$. Note that $\bar{\omega_2}\in \bigcap\pi``(U\cap T)$ and $\pi(U\cap T)=\pi``(U\cap T)$. Then $\bigcap\pi``(U\cap T)=\bigcap\pi(U\cap T)=\pi(\bigcap (U\cap T))\neq\emptyset$. So $\bigcap (U\cap T)\neq\emptyset$.
\end{proof}

So we can build a nontrivial embedding from $M$ to $M$ which contradicts the rigidity property of $M$.
\end{proof}

\begin{theorem}\label{characterization of srp above third level}
The following are equivalent:
\begin{enumerate}[(i)]
  \item $SRP^{L}(\gamma)$ holds for some $L$-cardinal $\gamma>\omega_2$.
  \item $0^{\sharp}$ exists.
  \item $SRP^{L}(\gamma)$ holds for every $L$-cardinal $\gamma\geq\omega_1$.
\end{enumerate}
\end{theorem}
\begin{proof}
$(i)\Rightarrow (ii)$ Assume $0^{\sharp}$  does not exist. Then $L$ satisfies all the conditions for $M$ in Theorem \ref{characterizaton of zero sharp thm}. From the proof of Theorem \ref{characterizaton of zero sharp thm} (replace $M$ with $L$), $SRP^{L}(\gamma)$ does not hold for any $L$ cardinal $\gamma>\omega_2$.

$(ii) \Rightarrow (iii)$ Note that if $X\prec H_{\kappa}$ and $\gamma\in X$, then $\mathcal{M}(0^{\sharp}, \gamma+1)\in X$ and its image under the transitive collapse of $X$ is $\mathcal{M}(0^{\sharp}, \bar{\gamma}+1)$.\footnote{$\mathcal{M}(0^{\sharp}, \alpha)$ is the unique transitive $(0^{\sharp},\alpha)$-model. For the notation, see \cite{Higherinfinite}.} Note that for $\alpha\in Ord, \mathcal{M}(0^{\sharp}, \alpha)\prec L$.
\end{proof}

So for $n\geq 3, SRP^{L}(\omega_n)$ is equivalent to $0^{\sharp}$ exists. We have characterized $SRP^{L}(\omega_n)$ for $n\geq 1$.

\begin{definition}
Suppose $M$ is an inner model. For $M$-cardinal $\lambda$, let $SRP^{M}_{<\lambda}(\lambda)$ denote the statement: for some regular cardinal $\theta>\lambda$, if $X\prec H_{\theta},|X|<\lambda$ and $\lambda\in X$, then $\bar{\lambda}$  is an $M$-cardinal.
\end{definition}

\begin{fact}\label{covering thm for mc}
(\cite[Theorem 1.3]{CoveringLemma})\quad Assume $0^{\dag}$ does not exist but there is an inner model with a measurable cardinal and $L[U]$ is chosen such that $\kappa=crit(U)$ is as small as possible. The one of the following holds:
      \begin{enumerate}[(a)]
        \item For every set $X$ of ordinals, there is a set $Y\in L[U]$ such that $Y\supseteq X$ and $|Y|=|X| + \omega_1$;
        \item There is a sequence $C\subseteq\kappa$, which is Prikry generic over $L[U]$, such that for all set $X$ of ordinals, there is a set $Y\in L[U,C]$ such that $Y\supseteq X$ and $|Y|=|X| + \omega_1$.
      \end{enumerate}
\end{fact}

\begin{fact}\label{fact on zero dagger}
(\cite[21.22 Exercise]{Higherinfinite}) \quad The following are equivalent:
\begin{enumerate}[(1)]
  \item $0^{\dag}$ exists.
  \item There is a $\kappa$-model for some $\kappa$ and an elementary embedding from that model to itself with critical point greater than $\kappa$.
\end{enumerate}
\end{fact}

\begin{theorem}\label{big thm about u}
Suppose there is an inner model with a measurable cardinal and $L[U]$ is chosen such that $\kappa=crit(U)$ is as small as possible. Suppose $\lambda>\kappa^{+}$ is an $L[U]$-cardinal. Then $SRP^{L[U]}_{<\lambda}(\lambda)$ if and only if $0^{\dag}$ exists.
\end{theorem}
\begin{proof}
$(\Rightarrow)$ We assume that $0^{\dag}$ does not exist and  try to get a contradiction. By Fact \ref{covering thm for mc}, we need to discuss two cases.

Case 1: Fact \ref{covering thm for mc}(a) holds. Let $\theta>\lambda$ be the witness regular cardinal for $SRP^{L[U]}_{<\lambda}(\lambda)$. Build an elementary chain $\langle Z_{\alpha}\mid\alpha<\kappa\rangle$ of submodels of $H_{\theta}$ such that for $\alpha<\beta<\kappa, Z_{\alpha}\prec Z_{\beta}\prec H_{\theta}, Z_{\alpha}\in Z_{\beta}, |Z_{\alpha}|=\kappa$ and $\{\kappa^{+}, \lambda\}\cup  tr(\{U\})\subseteq Z_0$.\footnote{In this article, $tr(X)$ stands for the transitive closure of $X$.} Let $Z=\bigcup_{\alpha<\kappa} Z_{\alpha}$. Then $|Z|=\kappa$. Let $\pi: N\cong Z\prec H_{\theta}$ and $\pi_{\alpha}: N_{\alpha}\cong Z_{\alpha}\prec H_{\theta}$ be the inverses of the collapsing maps. Since $Z_{\alpha}\prec Z$, let $j_{\alpha}: N_{\alpha}\prec N$ be the induced embedding. Then $\pi_{\alpha}=\pi\circ j_{\alpha}$ and $N=\bigcup_{\alpha<\kappa} j_{\alpha}``N_{\alpha}$. Let $crit(\pi)=\eta$. Then $\eta>\kappa=\bar{\kappa}$ and since $|Z|=\kappa, \eta\leq \bar{\kappa^{+}}$. So $\eta=\bar{\kappa^{+}}<\bar{\lambda}$. By $SRP^{L[U]}_{<\lambda}(\lambda)$, $\bar{\lambda}$ is an $L[U]$-cardinal. Let $W=\{X\subseteq\eta\mid X\in L[U]$ and $\eta\in\pi(X)$\}. Note that $U=\bar{U}\in N$ and $W\subseteq L_{\bar{\lambda}}[U]\subseteq N$. $W$ is $L[U]$-ultrafilter on $\eta$. Note that $Z\models ``|Z_{\alpha}|=\kappa$" and the image of $Z_{\alpha}$ under the transitive collapse of $Z$ is $j_{\alpha}``N_{\alpha}$. So for $\alpha<\kappa, j_{\alpha}``N_{\alpha}\in N$ and $N\models ``|j_{\alpha}``N_{\alpha}|=\kappa$".

\begin{lemma}\label{a key lemma on U}
$W$ is countably complete.
\end{lemma}
\begin{proof}
Suppose $Y\subseteq W$ and $Y$ is countable. We show that $\bigcap Y\neq\emptyset$. Since $Y\subseteq N$, take $\alpha<\kappa$ large enough such that $Y\subseteq j_{\alpha}``N_{\alpha}$. Let $S=\mathcal{P}(\eta)\cap L[U]\cap j_{\alpha}``N_{\alpha}$.  Note that $\mathcal{P}(\eta)\cap L[U]\in N$ and hence $S\in N$. $N\models |S|\leq\kappa$. Since Fact \ref{covering thm for mc}(a) holds  in $H_{\theta}$ and $N\prec H_{\theta}$, Fact \ref{covering thm for mc}(a) holds in $N$. Take $T\in N$ such that $T\subseteq \mathcal{P}(\eta)\cap L[U], T\supseteq S, T\in L[U]$ and $N\models |T|\leq\kappa$. Since $\eta>\kappa, \pi(T)=\pi``T$. Let $\bar{T}=\{X\in T\mid \eta\in \pi(X)\}$.
\begin{claim}
$\bar{T}\in N$.
\end{claim}
\begin{proof}
Since $N\models |T|\leq\kappa$, there is $h\in N$ such that $h: T\leftrightarrow \gamma$ for some $\gamma<\eta$. Then $\bar{T}=\{X\in T\mid \eta\in \pi``(h^{-1})(h(X))\}$. So $\bar{T}\in N$.
\end{proof}

Note that $\bigcap \bar{T}\neq\emptyset$ since $\pi(\bar{T})=\pi``\bar{T}$ and $\eta\in\bigcap \pi``\bar{T}=\bigcap \pi(\bar{T})=\pi(\bigcap \bar{T})$. Since $Y\subseteq S\subseteq T$ and $Y\subseteq W$, $Y\subseteq \bar{T}$ and hence $\bigcap Y\neq\emptyset$.
\end{proof}

So there exists a nontrivial elementary embedding $j: L[U]\prec L[U]$ with $crit(j)=\eta>\kappa$. By Fact \ref{fact on zero dagger}, $0^{\dag}$ exists. Contradiction.

Case 2: Fact \ref{covering thm for mc}(b) holds. The proof is essentially the same as Case 1 with small modifications (for example, let $tr(\{U,C\})\subseteq Z_0$ and $W=\{X\subseteq\eta\mid X\in L[U,C]$ and $\eta\in\pi(X)$\}).  Since Priky forcing preserves all cardinals, $\bar{\lambda}$ is an $L[U,C]$-cardinal. As in Case 1, we can show that there exists a nontrivial elementary
embedding $j: L[U,C]\prec L[U,C]$. Since $j(U,C)=(U,C), j\upharpoonright L[U]: L[U]\prec L[U]$. $crit(j\upharpoonright L[U])=\eta>\kappa$. So by Fact \ref{fact on zero dagger}, $0^{\dag}$ exists. Contradiction.

$(\Leftarrow)$ Assume $0^{\dag}$ exists. Suppose $\theta>\lambda$ is regular, $X\prec H_{\theta}, |X|<\lambda$ and $\lambda\in X$. We show that $\bar{\lambda}$ is an $L[U]$-cardinal. Since $\lambda\in X$ and $0^{\dag}\in X, \mathcal{M}(0^{\dag},\omega, \lambda+1)\in X$.\footnote{Note that $\mathcal{M}(0^{\dag},\omega, \alpha)$ is the unique transitive $(0^{\dag},\omega,\alpha)$-model. For the notation of $\mathcal{M}(0^{\dag},\omega, \alpha)$, see \cite{Higherinfinite}.} Note that for any $\alpha, \beta\in Ord, \mathcal{M}(0^{\dag},\alpha, \beta)\prec L[U]$. Since $\lambda$ is an $L[U]$-cardinal and $\lambda\in \mathcal{M}(0^{\dag},\omega, \lambda+1), \mathcal{M}(0^{\dag},\omega, \lambda+1)\models \lambda$ is a cardinal. Note that the image of $\mathcal{M}(0^{\dag},\omega, \lambda+1)$ under the transitive collapse of $X$ is $\mathcal{M}(0^{\dag}, \omega,\bar{\lambda}+1)$. So $\mathcal{M}(0^{\dag}, \omega,\bar{\lambda}+1)\models ``\bar{\lambda}$ is a cardinal". Since $\mathcal{M}(0^{\dag}, \omega,\bar{\lambda}+1)\prec L[U],\bar{\lambda}$ is an $L[U]$-cardinal.
\end{proof}

In \cite{Schindler3}, Thoralf R\"{a}sch and Ralf Schindler
introduced the condensation principle $\nabla_{\kappa}$: for any regular cardinal $\theta>\kappa, \{X\prec L_{\theta}\mid |X|<\kappa, X\cap\kappa\in\kappa$ and $L\models o.t.(X\cap \theta)$ is a cardinal\} is stationary. The notion of the strong reflecting property for $L$-cardinals was introduced before the author knew about the work on $\nabla_{\kappa}$ in \cite{Schindler3}. The following theorem summarizes the strength of $\nabla_{\omega_n}$ for $n\in\omega$.

\begin{theorem}\label{thm about remarkble cn}
\begin{enumerate}[(1)]
  \item (\cite[Theorem 2, 4]{Schindler3})\quad The following theories are equiconsistent:
\begin{enumerate}[(a)]
  \item $ZFC\,+\, \nabla_{\omega_1}$.
  \item $ZFC\,+\, \nabla_{\omega_2}$.
  \item $ZFC\,+$ there exists a remarkable cardinal.
\end{enumerate}
  \item \cite[Corollary 12]{Schindler3}\quad For $n\geq 3, \nabla_{\omega_n}$ is equivalent to $0^{\sharp}$ exists.
\end{enumerate}
\end{theorem}

Now we discuss the relationship between $SRP^{L}(\omega_n)$ and $\nabla_{\omega_n}$ for $n\in\omega$. By Theorem \ref{characterization of srp above third level} and \ref{thm about remarkble cn}, for $n\geq 3, SRP^{L}(\omega_n)$ is equivalent to $\nabla_{\omega_n}$. If $\kappa$ is regular cardinal and $\nabla_{\kappa}$ holds, then $\kappa$ is remarkable in $L$ (cf. \cite[Lemma 7]{Schindler3}). By Proposition \ref{srp of omega first}, $\nabla_{\omega_1}$ implies $SRP^{L}(\omega_1)$ which is strictly weaker.  By Theorem \ref{strength of omega two}, $SRP^{L}(\omega_2)$ does not imply $\nabla_{\omega_2}$ since $\nabla_{\omega_2}$ implies $\omega_2$ is remarkable in $L$. By Theorem \ref{thm about remarkble cn} and \ref{strength of omega two}, the strength of $SRP^{L}(\omega_2)$ is strictly stronger than $\nabla_{\omega_2}$.

In Definition \ref{definition of strong reflecting property and weakly}, we only consider countable elementary submodels of $H_{\kappa}$. Similarly as $\nabla_{\kappa}$ we could also consider uncountable elementary submodels of $H_{\kappa}$. However this does not change the picture. Obviously, $SRP^{L}_{<\omega_1}(\omega_1)$ \text{if{f}} $SRP^{L}(\omega_1)$. By Proposition \ref{transive collapse of L cardinal}, $SRP^{L}_{<\omega_2}(\omega_2)$ \text{if{f}} $SRP^{L}(\omega_2)$. By Theorem \ref{characterizaton of zero sharp thm}, for $n\geq 3, SRP^{L}_{<\omega_n}(\omega_n)$ \text{if{f}} $0^{\sharp}$ exists \text{if{f}} $SRP^{L}(\omega_n)$.

\section{Harrington's Principle $HP(L)$ and its generalization}

In this section, we define the generalized Harrington's Principle $HP(M)$ for any inner model $M$. Considering various known examples of inner models we give particular characterizations of $HP(M)$, while we also show that in some cases this
generalized principle fails.

Recall that for limit ordinal $\alpha>\omega$, $\alpha$ is $x$-admissible if and only if there is no $\Sigma_1(L_{\alpha}[x])$ mapping from an ordinal $\delta<\alpha$ cofinally into $\alpha$ (see \cite[Lemma 7.2]{Constructiblity}).

\begin{definition}\label{def of HP}
Suppose $M$ is an inner model. The Generalized Harrington's Principle $HP(M)$ denotes the following statement: there is a real $x$ such that, for any ordinal $\alpha$, if $\alpha$ is $x$-admissible then $\alpha$ is an $M$-cardinal, i.e., $M\models \alpha$\textrm{ is a cardinal}. $HP(L)$ denotes Harrington's Principle.
\end{definition}

Harrington's principle $HP(L)$ was isolated by Harrington in the proof of his celebrated theorem ``$Det(\Sigma^1_1)$ implies $0^{\sharp}"$ in \cite{Harrington}.

\begin{fact}\label{fact on zero sharp in devlin}
(Essentially \cite{Constructiblity})\quad $(Z_4)$ \quad $L_{\omega_2}$ has an uncountable set of indiscernibles if and only if $0^{\sharp}$ exists.
\end{fact}

\begin{theorem}\label{cited theorem from Woodin}
$(Z_4)$ \quad  The following are equivalent:\footnote{In \cite{YongCheng2}, we define $0^{\sharp}$ as the minimal iterable mouse and prove in $Z_4$ that $HP(L)$ is equivalent to $0^{\sharp}$ exists. Theorem \ref{cited theorem from Woodin} proves that these two definitions of $0^{\sharp}$ are equivalent in  $Z_4$.}
\begin{enumerate}[(1)]
  \item $HP(L)$.
  \item $L_{\omega_2}$ has an uncountable set of indiscernibles.
  \item $0^{\sharp}$ exists.
\end{enumerate}
\end{theorem}
\begin{proof}
Note that in $Z_2, 0^{\sharp}$ implies $HP(L)$ since any $0^{\sharp}$-admissible ordinal is an $L$-cardinal. It suffices to show that $(1)\Rightarrow (2)$. Let $a$ be the witness real for $HP(L)$. We work in $L[a]$. Pick $\eta>\omega_2$  and $N$ such that $\eta$ is $a$-admissible, $N\prec L_{\eta}[a],\omega_2\in N, |N|=\omega_1$ and $N$ is closed under $\omega$-sequences. Let $j: L_{\theta}[a]\cong N\prec L_{\eta}[a]$ be the inverse of the collapsing map and $\kappa=crit(j)$. By $HP(L), \theta$ is an $L$-cardinal.  Define $U=\{X\subseteq \kappa\mid X\in L\wedge\kappa\in j(X)\}$. Note that $(\kappa^{+})^{L}\leq\theta<\omega_2$ and $U\subseteq L_{\theta}$ is an $L$-ultrafilter on $\kappa$. Do the ultrapower construction for $\langle L_{\omega_2}, \in, U\rangle$. Since $L_{\theta}[a]$ is closed under $\omega$-sequences,$L_{\omega_2}/U$ is well founded and hence we get a nontrivial elementary embedding $e: L_{\omega_2}\prec L_{\omega_2}$ with $crit(e)=\kappa$.

Now we show that there exists a club on $\omega_2$ of regular $L$-cardinals. Suppose $X\prec L_{\eta}[a], \omega_1\subseteq X$ and $\omega_2\in X$. The transitive collapse of $X$ is $L_{\bar{\eta}}[a]$ for some $\bar{\eta}$. Since $L_{\eta}\models \omega_2$ is a regular cardinal, $L_{\bar{\eta}}\models \bar{\omega_2}$ is a regular cardinal. By $HP(L)$, $\bar{\eta}$ is an $L$-cardinal and hence $\bar{\omega_2}$ is a regular $L$-cardinal. Since $\omega_1\subseteq X$, $\bar{\omega_2}=X\cap \omega_2$.  We have shown that if $X\prec L_{\eta}[a], \omega_1\subseteq X$ and $\omega_2\in X$, then $X\cap \omega_2=\bar{\omega_2}$ is a regular $L$-cardinal. So there exists a club on $\omega_2$ of regular $L$-cardinals. Let $D$ be such a club such that $D\cap(\kappa+1)=\emptyset$.

\begin{claim}\label{equlity claim}
For any $\alpha\in D, e(\alpha)=\alpha$.
\end{claim}
\begin{proof}
Suppose $\alpha\in D$ and $f\in L_{\omega_2}$ where $f: \kappa\rightarrow\alpha$. Since $\alpha>\kappa$ is a regular $L$-cardinal, $f$ is bounded by some $\eta<\alpha$. So $[f]<[c_{\eta}]$. Hence $e(\alpha)=lim_{\beta\rightarrow\alpha} e(\beta)$. If $\beta<\alpha$, then $|e(\beta)|\leq (|\beta^{\kappa}|)^{L}\leq\alpha$. So $e(\alpha)=\alpha$.
\end{proof}

We define a sequence $\langle C_\alpha: \alpha<\omega_1\rangle$ as follows. Let $C_{0}=D$.
For any $\nu<\omega_1, C_{\nu+1}=\{\mu\in C_{\nu}\mid \mu$ is the $\mu$-th element of $C_{\nu}$ in the increasing enumeration of $C_{\nu}$\}. If $\nu\leq\omega_1$ is a limit ordinal, $C_{\nu}=\bigcap_{\beta<\nu} C_{\beta}$. Note that $C_{\nu}$ is a club on $\omega_2$ for all $\nu\leq\omega_1$. By Claim \ref{equlity claim}, for $\nu\leq\omega_1, e\upharpoonright C_{\nu}=id$. Now we will find $\omega_1$-many indiscernibles for $(L_{\omega_2}, \in)$. The rest of the argument essentially follows
from \cite[Theorem 18.20]{Jech}.

For each $\nu<\omega_1$, let $M_{\nu}$ be the Skolem hull of $\kappa\cup C_{\nu}$ in $L_{\omega_2}$. The transitive collapse of $M_{\nu}$ is $L_{\omega_2}$. Let $i_{\nu}: L_{\omega_2}\cong M_{\nu}\prec L_{\omega_2}$ be the inverse of the collapsing map and  $\kappa_{\nu}=i_{\nu}(\kappa)$. By \cite[Lemma 18.24,18.25, 18.26]{Jech}, $\{\kappa_{\nu}\mid \nu<\omega_1\}$ is a set of indiscernibles for $L_{\omega_2}$.\footnote{Note that the proof of \cite[Theorem 18.20]{Jech}, as opposed to the proof of Theorem \ref{cited theorem from Woodin} above, is not done in $Z_4$.}
\end{proof}

\begin{theorem}\label{main theorem by chengyong}
(\cite{YongCheng2})\quad $Z_3\, +\, HP(L)$ does not imply $0^{\sharp}$ exists.
\end{theorem}

By a similar argument as in Theorem \ref{cited theorem from Woodin} we can show from $Z_3\, +\, HP(L)$ that there exists a nontrivial elementary embedding $j: L_{\omega_1}\prec L_{\omega_1}$ and there is a club $C\subseteq \omega_1$ of regular $L$-cardinals. However, by Theorem \ref{main theorem by chengyong}, from these we can not prove in $Z_3$ that $0^{\sharp}$ exists.

Note that Theorem 3.3 still holds if we replace the term ``$L$-cardinal" with
any large cardinal notion compatible with $L$ in the definition of $HP(L)$. This is because the Silver indiscernibles can have any large cardinal
property compatible with $L$.\footnote{Examples of large cardinal notions compatible with $L$: inaccessible cardinal,reflecting cardinal, Mahlo cardinal, weakly compact, indescribable cardinal, unfoldable cardinal,  subtle cardinal, ineffable cardinal, 1-iterable cardinal, remarkable cardinal, 2-iterable cardinal and $\omega$-Erd$\ddot{o}$s cardinal.}

\begin{fact}\label{fact on zero dagger two}
(\cite[Theorem 21.15]{Higherinfinite}) \quad The following are equivalent:
\begin{enumerate}[(1)]
  \item $0^{\dag}$ exists.
  \item For every uncountable cardinal $\kappa$ there is a $\kappa$-model and a double class $\langle X,Y\rangle$ of indiscernibles for it such that: $X\subseteq\kappa$ is closed unbounded, $Y\subseteq Ord\setminus(\kappa+1)$ is a closed unbounded class, $X\cup\{\kappa\}\cup Y$ contains every uncountable cardinal and the Skolem hull of $X\cup Y$ in the $\kappa$-model is again the model.
\end{enumerate}
\end{fact}

\begin{fact}\label{condensation for L[U]}
(\cite[Lemma 1.7]{begininnermodel})\quad Suppose that $A$ is a set, $X\prec L_{\alpha}[A]$ where $\alpha\in Ord\cup \{Ord\}$ and the transitive closure of $A\cap L_{\alpha}[A]$ is contained in $X$. Then $X\cong L_{\alpha^{\prime}}[A]$ for some $\alpha^{\prime}\leq\alpha$.
\end{fact}

\begin{fact}\label{fact on zero dagger lage cardinal}
(Folklore)\quad Suppose $0^{\dag}$ exists, $L[U]$ is the unique $\kappa$-model and $\langle X,Y\rangle$ is the double class of indiscernibles for $L[U]$ as in Fact \ref{fact on zero dagger two}. If $\alpha\leq\kappa$ is $0^{\dag}$-admissible, then $X$ is unbounded in $\alpha$, and if $\alpha>\kappa$ is $0^{\dag}$-admissible, then $Y$ is unbounded in $\alpha$.\footnote{I would like to thank W.Hugh Woodin and Sy Friedman for pointing out this fact to me. The proof of this fact is essentially similar as the proof of the following standard fact: if $0^{\sharp}$ exists, $I$ is the class of Silver indiscernibles and $\alpha$ is $0^{\sharp}$-admissible, then $I$ is unbounded in $\alpha$ (see \cite[Theorem 4.3]{SyFriedman}).}
\end{fact}

\begin{theorem}\label{HP for measurable cn}
Suppose $\kappa$ is a measurable cardinal and $L[U]$ is the unique $\kappa$-model. Then $HP(L[U])$ if and only if $0^{\dag}$ exists.
\end{theorem}
\begin{proof}
$(\Rightarrow)$ Let $x$ be the witness real for $HP(L[U])$. Pick $\lambda> 2^{\kappa}$ and $X$ such that $\lambda$ is $(x,U)$-admissible,  $X\prec L_{\lambda}[U][x]$, $|X|=2^{\kappa}$, $X$ is closed under $\omega$-sequences and the transitive closure of $U\cap L_{\lambda}[U]$ is contained in $X$. By Fact \ref{condensation for L[U]}, the transitive collapse of $X$ is of the form  $L_{\theta}[U][x]$. Let $j: L_{\theta}[U][x]\cong X\prec L_{\lambda}[U][x]$ be the inverse of the collapsing map and $\eta=crit(j)$. Note that $\eta>\kappa$. Since $\theta$ is $(x,U)$-admissible, by $HP(L[U])$, $\theta$ is an $L[U]$-cardinal. Define $\bar{U}=\{X\subseteq\eta\mid X\in L[U]$ and $\eta\in j(X)\}$. Since $(\eta^{+})^{L[U]}\leq\theta, \bar{U}\subseteq L_{\theta}[U]$. $\bar{U}$ is an $L[U]$-ultrafilter on $\eta$. Since $L_{\theta}[U][x]$ is closed under $\omega$-sequences, $\bar{U}$ is countably complete. So we can build a nontrivial embedding from $L[U]$ to $L[U]$ with critical point greater than $\kappa$. By Fact \ref{fact on zero dagger}, $0^{\dag}$ exists.

$(\Leftarrow)$ Suppose $0^{\dag}$ exists and $\alpha$ is $0^{\dag}$-admissible. We show that $\alpha$ is an $L[U]$-cardinal. By Fact \ref{fact on zero dagger two}, let $\langle X,Y\rangle$ be the double class of indiscernibles for $L[U]$. If $\alpha\leq\kappa$, then by Fact \ref{fact on zero dagger lage cardinal}, $\alpha\in X$. If $\alpha>\kappa$, then by Fact \ref{fact on zero dagger lage cardinal}, $\alpha\in Y$. Trivially, elements of $X$ and $Y$ are $L[U]$-cardinals.
\end{proof}

\begin{fact}\label{embedding not exist}
(\cite{CoveringLemma}, \cite{Steel})\quad Suppose there is no inner model with one measurable cardinal and let $K$ be the corresponding core model. Then, $K$ has the rigidity property.
\end{fact}

\begin{corollary}\label{Theorem for core model}
\begin{enumerate}[(1)]
      \item Suppose $0^{\sharp}$ exists. Then $HP(L[0^{\sharp}])$ if and only if $(0^{\sharp})^{\sharp}$ exists.
          \item Suppose there is no inner model with one measurable cardinal and that $K$ is the corresponding core model. Then $HP(K)$ does not hold.
\end{enumerate}
\end{corollary}
\begin{proof}
$(1)$ Follows from the proof of $``HP(L)\Leftrightarrow 0^{\sharp}$ exists". Note that if $\alpha$ is $(0^{\sharp})^{\sharp}$-admissible and $I$ is the class of Silver indiscernibles for $L[0^{\sharp}]$, then $I$ is unbounded in $\alpha$ and hence $\alpha\in I$.

(2) Note that $K=L[\mathcal{M}]$ where $\mathcal{M}$ is a class of mice. Suppose $HP(K)$ holds and $x$ is the witness real for $HP(K)$. Pick  $\theta>\omega_2$ and $X$ such that $\theta$ is $(\mathcal{M}, x)$-admissible, $X\prec J_{\theta}[\mathcal{M}, x]$, $\omega_2\in X, |X|=\omega_1$ and $X$ is closed under $\omega$-sequences. Since $K\models GCH$, such an $X$ exists. By the condensation theorem for  $K$, let $j: J_{\theta^{\prime}}[\mathcal{M}\upharpoonright \theta^{\prime}, x]\cong X \prec J_{\theta}[\mathcal{M}, x]$ be the inverse of the collapsing map. Let $\lambda=crit(j)$ and $U=\{X\subseteq\lambda\mid X\in K$ and $\lambda\in j(X)$\}. Note that $\theta^{\prime}$ is a $K$-cardinal and $U$ is a countably complete $K$-ultrafilter on $\lambda$. So there is a nontrivial elementary embedding from $K$ to $K$ which contradicts Fact \ref{embedding not exist}.
\end{proof}

From proof of Corollary \ref{Theorem for core model}(2), if $M$ is an $L$-like inner model, $M$ has the rigidity property and some proper form of condensation, and $M\models CH$, then $HP(M)$ does not hold.

\begin{fact}\label{fact about HOD}
(\cite{Steel})\quad $(AD^{L(R)})\quad \mathsf{HOD}^{L(R)}=L(P)$ for some $P\subseteq \Theta$ where $\Theta=\sup\{\alpha\mid \exists f\in L(R)(f:R\rightarrow\alpha$ is surjective$)\}$.
\end{fact}

It is an open question whether there exists a nontrivial elementary embedding from $\mathsf{HOD}$ to $\mathsf{HOD}$.\footnote{The answer to this question is negative if $V=\mathsf{HOD}$.\cite[Theorem 21]{Hamkins} provides a very easy proof of the Kunen inconsistency in the case $V=\mathsf{HOD}$.}  However, the following fact shows that the answer to this question is negative for embeddings which are definable in $V$ from parameters.

\begin{fact}\label{definable HOD}
(\cite[Theorem 35]{Hamkins})\quad Do not assume $AC$. There is no nontrivial elementary embedding from $\mathsf{HOD}$ to $\mathsf{HOD}$ that is definable in $V$ from parameters.
\end{fact}

\begin{theorem}\label{HP for HOD}
$(ZF+ AD^{L(R)})\quad \mathsf{HP(HOD)}$ does not hold.
\end{theorem}
\begin{proof}
By Fact \ref{fact about HOD}, under $ZF + AD^{L(R)}, \mathsf{HOD}=L(P)$ for some $P\subseteq \Theta$. Suppose $\mathsf{HP(HOD)}$ holds. Then since $L(P)\models CH$, by a similar proof as in Corollary \ref{Theorem for core model}(2) we can show that there exists a nontrivial elementary embedding $j: L(P)\rightarrow L(P)$. Note that $j$ is definable in $V$ from parameters. i.e. there is a formula $\varphi$ and parameter $\vec{a}$ such that $j(x)=y$ if and only if $\varphi(x,y,\vec{a})$. This contradicts Fact \ref{definable HOD}.
\end{proof}

\section{Relationship between $HP(L)$ and the strong reflecting property for $L$-cardinals}

In this section, we discuss the relationship between the strong reflecting property for $L$-cardinals and  Harrington's Principle $HP(L)$.

\begin{theorem}\label{compare first level thm}
(Set forcing) \quad $SRP^{L}(\omega_1)$ implies $Con(Z_2\,+ \,HP(L))$.
\end{theorem}
\begin{proof}
Suppose $SRP^{L}(\omega_1)$ holds and we want to build a model of $Z_2\,+ \,HP(L)$. By Proposition \ref{srp of omega first}, $\omega_1$ is limit cardinal in $L$. i.e. $\{\alpha<\omega_1\mid \alpha$ is an $L$-cardinal\} is a club. Let $C=\{\omega\leq\alpha<\omega_{1}\mid \alpha$ is an $L$-cardinal and $L_{\alpha}\prec L_{\omega_1}\}$. Note that $C$ is a club.  Let \[D=\{\gamma<\omega_1\mid  (L_{\gamma}[C], C\cap\gamma)\prec (L_{\omega_1}[C], C)\}.\]
Note that $D\subseteq C$. Define  $F: \omega^{\omega}\rightarrow \omega^{\omega}$ as follows: if $y\subseteq\omega$ codes $\gamma$, then $F(y)$ is a real which codes $(\beta, C\cap\beta)$ where $\beta$ is the least element of $D$ such that $\beta>\gamma$ (since $D$ is a club in $\omega_1$, such a $\beta$ exists); if $y$ does not code an ordinal, let $F(y)=\emptyset$.

Let $\langle\delta_{\alpha}\mid \alpha<\omega_1\rangle$ be a pairwise almost disjoint set of reals such that $\delta_{\alpha}$ is the $<_{L[C]}$-least real which is almost disjoint from any member of $\{\delta_{\beta}\mid \beta<\alpha\}$ and $\langle\delta_{\nu}\mid \nu<\omega\rangle\in L_{\alpha}$ for every admissible ordinal $\alpha<\omega_1$.

Let $\langle x_{\alpha}\mid \alpha<\omega_1\rangle$ be the enumeration of $\mathcal{P}(\omega)$ in  $L[C]$ in the order of construction. Let $Z_F \subseteq\omega_1$ be defined as:
\[Z_{F}=\{\alpha\cdot\omega+i\mid   \alpha<\omega_1\wedge i\in F(x_{\alpha}) \}.\]
Now we do almost disjoint forcing to code $Z_{F}$  via $\langle\delta_{\alpha}\mid \alpha<\omega_1\rangle$. Then we get a real $x$ such that $\alpha\in Z_{F}\Leftrightarrow |x\cap \delta_{\alpha}|<\omega$. The forcing is $c.c.c$ and hence preserves all cardinals.

Now we work in $L[x]$. Take the least $\theta$ such that $L_{\theta}[x]\models Z_{2}$. We will show that $L_{\theta}[x]\models HP(L)$. By absoluteness, it suffices to show that if $\alpha<\theta$ is $x$-admissible, then $\alpha$ is an $L$-cardinal. Fix some $x$-admissible $\alpha<\theta$ and let
\[\gamma_0=\sup(\alpha\cap D).\]

If $\alpha\cap D=\emptyset$, let $\gamma_0=0$.   Note that if $\gamma_0>0$, then $\gamma_0\in D$. We assume that $\gamma_0<\alpha$ and try to get a contradiction. Let $\alpha_0$ be the least admissible ordinal such that $\alpha_0>\gamma_0$. Since $\alpha$ is admissible, $\alpha_0\leq\alpha$.

\begin{claim}\label{show that}
$C\cap\alpha_0=C\cap (\gamma_0+1).$
\end{claim}
\begin{proof}
We show that $C\cap\alpha_0\subseteq C\cap (\gamma_0+1)$. Suppose $\gamma\in C\cap\alpha_0$ and $\gamma>\gamma_0$. Since $\gamma\in C, L_{\gamma}\prec L_{\omega_1}$. Since $\alpha_0$ is definable from $\gamma_0$, it follows that $\alpha_0$ is definable in $L_{\gamma}$. So  $\alpha_0\leq\gamma$. Contradiction.
\end{proof}

By Claim \ref{show that}, $L_{\alpha_0}[C]=L_{\alpha_0}[C\cap\gamma_0]$. We need the following lemma to get that $L_{\gamma_0}[C\cap\gamma_0][x]=L_{\gamma_0}[x]$ in Claim \ref{key calim on countable}.

\begin{lemma}\label{key claim in step}
$C\cap\gamma_0\in L_{\gamma_0+1}[x].$
\end{lemma}
\begin{proof}
We prove by induction that for any $\gamma\in D\cap\theta, C\cap\gamma\in L_{\gamma+1}[x]$. Fix $\gamma\in D\cap\theta$. Suppose for any $\gamma^{\prime}\in D\cap\gamma$, $C\cap\gamma^{\prime}\in L_{\gamma^{\prime}+1}[x]$. We show that $C\cap\gamma\in L_{\gamma+1}[x]$.

Case 1: There is $\gamma^{\prime}\in D$ such that $\gamma$ is the least element of $D$ such that $\gamma>\gamma^{\prime}$. Let $\eta$ be the least admissible ordinal such that $\eta>\gamma^{\prime}$. By a similar argument as in Claim \ref{show that}, $C\cap\eta=C\cap(\gamma^{\prime}+1)$. From our definitions, for any $\beta<\eta$ we have: (1) $\langle x_{\xi}\mid\xi\in\beta\rangle\in L_{\eta}[C]=L_{\eta}[C\cap\gamma^{\prime}]$; (2) $\langle\delta_{\xi}\mid\xi\in\beta\rangle\in L_{\eta}[C]=L_{\eta}[C\cap\gamma^{\prime}]$; and (3) $\langle x_{\xi}\mid\xi\in\eta\rangle$ enumerates $\mathcal{P}(\omega)\cap L_{\eta}[C]=\mathcal{P}(\omega)\cap L_{\eta}[C\cap\gamma^{\prime}]$.

Suppose $y\subseteq\omega$ and $y\in L_{\eta}[C\cap\gamma^{\prime}]$. Then $y=x_{\xi}$ for some $\xi<\eta$. Note that $\xi\cdot\omega+i<\eta$ for any $i<\omega$. Moreover, $i\in F(y)$ if and only if $|x\cap\delta_{\xi\cdot\omega+i}|<\omega$. So $F(y)\in L_{\eta}[C\cap\gamma^{\prime}][x]$. Hence we have shown that if $y\in\mathcal{P}(\omega)\cap L_{\eta}[C\cap\gamma^{\prime}]$, then $F(y)\in L_{\eta}[C\cap\gamma^{\prime}, x]$.

\begin{claim}\label{first key calim on countable}
$L_{\eta}[C\cap\gamma^{\prime}]\models\gamma^{\prime}<\omega_1$.
\end{claim}
\begin{proof}
Suppose, towards a contradiction, that
\begin{equation}\label{key equation in sec two}
\text{$\gamma^{\prime}=\omega_1^{L_{\eta}[C\cap\gamma^{\prime}]}$.}
\end{equation}
Let $P$ be the almost disjoint forcing that codes $Z_{F}$ via the almost disjoint system $\langle\delta_\beta \mid \beta<\omega_1\rangle$.\footnote{$P=[\omega]^{<\omega}\times [Z_{F}]^{<\omega}$. $(p,q)\leq (p^{\prime},q^{\prime})$ \text{if{f}} $p\supseteq  p^{\prime}, q\supseteq q^{\prime}$ and $\forall\alpha\in q^{\prime}(p\cap \delta_{\alpha}\subseteq p^{\prime})$.} From our definitions of $C, F$ and $\langle x_{\alpha}\mid \alpha<\omega_1\rangle$, $P$ is a definable subset of $L_{\omega_1}[C]$. Standard argument gives that $P$ is $\omega_1$-c.c. in $L_{\omega_1}[C]$.\footnote{i.e. If $D\subseteq P$ is a maximal antichain with $D\in L_{\omega_1}[C]$, then $L_{\omega_1}[C]\models D$ is at most countable.}  Let $P^{\ast}=P\cap L_{\gamma^{\prime}}[C]$. Since $\gamma^{\prime}\in D$,
\begin{equation}\label{property of gamma zero}
(L_{\gamma^{\prime}}[C], C\cap\gamma^{\prime})\prec (L_{\omega_1}[C], C).
\end{equation}
Suppose $D^{\ast}\subseteq P^{\ast}$ is a maximal antichain with $D^{\ast}\in L_{\gamma^{\prime}}[C]$. Then by (\ref{property of gamma zero}), $D^{\ast}$ is a maximal antichain in $P$. Since $L_{\omega_1}[C]\models D^{\ast}$ is at most countable, by (\ref{property of gamma zero}), $L_{\gamma^{\prime}}[C]\models D^{\ast}$ is at most countable. So $P^{\ast}$ is $\omega_1$-c.c. in $L_{\gamma^{\prime}}[C]$.
By (\ref{key equation in sec two}),
\begin{equation}\label{equation on two}
\text{$L_{\eta}[C\cap\gamma^{\prime}]\cap 2^{\omega}=L_{\gamma^{\prime}}[C\cap\gamma^{\prime}]\cap 2^{\omega}$.}
\end{equation}
Since $P^{\ast}$ is $\omega_1$-c.c. in $L_{\gamma^{\prime}}[C]$, by (\ref{equation on two}), $P^{\ast}$ is $\omega_1$-c.c in $L_{\eta}[C\cap\gamma^{\prime}]$.

We show that $x$ is generic over $L_{\eta}[C\cap\gamma^{\prime}]$ for $P^{\ast}$. Let $Y\subseteq P^{\ast}$ be a maximal antichain with $Y\in L_{\eta}[C\cap\gamma^{\prime}]$. Since $P^{\ast}$ is $\omega_1$-c.c in $L_{\eta}[C\cap\gamma^{\prime}]$, by (\ref{key equation in sec two}),    $Y\in L_{\gamma^{\prime}}[C\cap\gamma^{\prime}]$. By (\ref{property of gamma zero}), $Y$ is a maximal antichain in $P$. So the filter given by $x$ meets $Y$.

Note that $\gamma^{\prime}=\omega_1^{L_{\eta}[C\cap\gamma^{\prime}]}=\omega_1^{L_{\eta}[C\cap\gamma^{\prime}][x]}$. Since $\gamma^{\prime}\in D$, by induction hypothesis $L_{\gamma^{\prime}}[C\cap\gamma^{\prime}, x]=L_{\gamma^{\prime}}[x]$. So  $L_{\gamma^{\prime}}[x]\models Z_2$ which contradicts the minimality of $\theta$.
\end{proof}

Take $y\in L_{\eta}[C\cap\gamma^{\prime}]\cap\mathcal{P}(\omega)$ such that $y$ codes $\gamma^{\prime}$. So $F(y)$ codes $(\gamma,  C\cap\gamma)$ and $F(y)\in L_{\eta}[C\cap\gamma^{\prime}, x]$. Then $F(y)$ is definable in $L_{\gamma}[C\cap\gamma^{\prime}, x]$. By induction hypothesis, $F(y)\in L_{\gamma+1}[x]$. Since $F(y)$ codes $C\cap\gamma$, $C\cap\gamma\in L_{\gamma+1}[x]$.

Case 2: $\gamma$ is the least element of $D$. Take $y\in L_{\omega}[C]\cap\mathcal{P}(\omega)$ such that $y$ codes $0$. Then $y=x_0$. Since $\gamma$ is the least element of $D$ such that $\gamma>0$, $F(y)$ codes $C\cap\gamma$. Note that for any $\beta<\omega,\langle\delta_{\xi}\mid\xi\in\beta\rangle\in L_{\omega}[C]$ and $i\in F(y)$ if and only if $|x\cap\delta_{i}|$ is finite. So
$F(y)$ is definable in $L_{\omega}[x,C]$. Since $C\cap \omega=\emptyset$,
$F(y)\in L_{\gamma+1}[x]$. Since $F(y)$ codes $C\cap\gamma$, $C\cap\gamma\in L_{\gamma+1}[x]$.

Case 3: $\gamma$ is a limit point of $D$. Then a standard argument gives that $C\cap\gamma\in L_{\gamma+1}[x]$ by induction hypothesis.

Since $\gamma_0\in D\cap\theta$, we have $C\cap\gamma_0\in L_{\gamma_0+1}[x]$.
\end{proof}

\begin{claim}\label{key calim on countable}
$\gamma_0$ is countable in $L_{\alpha_0}[C\cap\gamma_0]$.
\end{claim}
\begin{proof}
The proof is essentially the same as Claim \ref{first key calim on countable}  \,(replace $\eta$ by $\alpha_0$ and $\gamma^{\prime}$ by $\gamma_0$). Suppose, towards a contradiction, that $\gamma_0=\omega_1^{L_{\alpha_0}[C\cap\gamma_0]}$.
By the similar argument as Claim \ref{first key calim on countable}, we can show that $x$ is generic over $L_{\alpha_0}[C\cap\gamma_0]$ for $P^{\ast}=P\cap L_{\gamma_0}[C]$.\footnote{$P$ is the almost disjoint forcing that codes $Z_{F}$ via $\langle\delta_{\beta}\mid \beta<\omega_1\rangle$.} Since $\gamma_0=\omega_1^{L_{\alpha_0}[C\cap\gamma_0]}=\omega_1^{L_{\alpha_0}[C\cap\gamma_0][x]}$ and by Lemma \ref{key claim in step}, $L_{\gamma_0}[C\cap\gamma_0][x]=L_{\gamma_0}[x]$, we have $L_{\gamma_0}[x]\models Z_2$ which contradicts the minimality of $\theta$.
\end{proof}

From our definitions, we have:
\begin{equation}\label{rule for adjs}
\text{For $\eta<\alpha_0, \langle\delta_{\beta}: \beta<\eta\rangle\in L_{\alpha_0}[C]=L_{\alpha_0}[C\cap\gamma_0]$;}
\end{equation}
\begin{equation}\label{euma rule}
\text{$\langle x_{\alpha}\mid \alpha<\alpha_0\rangle$ enumerates $\mathcal{P}(\omega)\cap L_{\alpha_0}[C]=\mathcal{P}(\omega)\cap L_{\alpha_0}[C\cap\gamma_0]$.}
\end{equation}

\begin{claim}\label{countable claim}
If $y\in \mathcal{P}(\omega)\cap  L_{\alpha_0}[C\cap\gamma_0]$, then $F(y)\in L_{\alpha_0}[x]$.
\end{claim}
\begin{proof}
Suppose $y\in \mathcal{P}(\omega)\cap  L_{\alpha_0}[C\cap\gamma_0]$. By \eqref{euma rule}, $y=x_{\xi}$ for some $\xi<\alpha_0$. Note that for $\xi<\alpha_0, \xi\cdot\omega+i<\alpha_0$ for any $i\in\omega$. By the definition of $Z_{F}, i\in F(y)\Leftrightarrow \xi\cdot\omega+i\in Z_{F}\Leftrightarrow |x\cap \delta_{\xi\cdot\omega+i}|<\omega$. By \eqref{rule for adjs}, $F(y)\in L_{\alpha_0}[C\cap\gamma_0][x]$. Since $C\cap\gamma_0\in L_{\gamma_0+1}[x]$ by Lemma \ref{key claim in step}, we have $L_{\alpha_0}[C\cap\gamma_0][x]=L_{\alpha_0}[x]$. So $F(y)\in L_{\alpha_0}[x]$.
\end{proof}

By Claim \ref{key calim on countable}, there exists a real $y\in L_{\alpha_0}[C\cap\gamma_0]$ such that $y$ codes $\gamma_0$. Note that $F(y)$ codes $\gamma_1$ where $\gamma_1$ is the least element of $C$ such that $\gamma_1>\gamma_0$ and $(L_{\gamma_1}[C], C\cap\gamma_1)\prec (L_{\omega_1}[C], C)$. Since $F(y)$ codes $\gamma_1$ and $F(y)\in L_{\alpha_0}[x]$, $\gamma_1<\alpha_0$. Since $\gamma_1<\alpha$ and $(L_{\gamma_1}[C], C\cap\gamma_1)\prec (L_{\omega_1}[C], C)$, by the definition of $\gamma_0$, we have that $\gamma_1\leq\gamma_0$. Contradiction.

So the assumption $\gamma_0<\alpha$ is false. Then $\gamma_0=\alpha$. So $\alpha \in C$ and hence $\alpha$ is an $L$-cardinal. We have shown  that $L_{\theta}[x]\models Z_2\,+ \,HP(L)$.
\end{proof}

\begin{theorem}\label{result on class forcing}
(\cite[Theorem 3.1, 3.2]{YongCheng2}) \quad (Class forcing) \quad $Z_2\,+ \,HP(L)$ is equiconsistent with $ZFC$ and $Z_3\,+ \,HP(L)$ is equiconsistent with $ZFC\, +$  there exists a remarkable cardinal.
\end{theorem}

\begin{corollary}
\begin{enumerate}[(a)]
  \item For $n\geq 3, SRP^{L}(\omega_n)$ is equivalent to $HP(L)$.
  \item (Set forcing) \quad $SRP^{L}(\omega_2)$ is strictly stronger than $Z_3\,+ \,HP(L)$.
      \item (Set forcing) \quad $SRP^{L}(\omega_1)$ is strictly stronger than $Z_2\,+ \,HP(L)$.
\end{enumerate}
\end{corollary}
\begin{proof}
(a) follows from Theorem \ref{characterization of srp above third level} and Theorem \ref{cited theorem from Woodin}. (b) follows from Theorem \ref{strength of omega two} and Theorem \ref{result on class forcing}. (c) follows from Theorem \ref{compare first level thm}, Theorem \ref{result on class forcing} and Proposition \ref{srp of omega first}.
\end{proof}

\end{document}